\documentclass[preprint]{elsarticle}
\usepackage{amssymb}
\usepackage{amsthm}
\usepackage{hyperref}
\usepackage{tikz}

\newtheorem{theorem}{Theorem}
\newtheorem{lemma}[theorem]{Lemma}
\newtheorem{proposition}[theorem]{Proposition}
\newtheorem{corollary}[theorem]{Corollary}

\theoremstyle{definition}
\newtheorem{definition}[theorem]{Definition}

\newtheorem{remark}[theorem]{Remark}

\begin{document}
	
\begin{frontmatter}

\title{The Frobenius number for sequences of triangular and tetrahedral numbers}

\author[AMRP]{A.M.~Robles-P\'erez\corref{cor1}\fnref{fn1}}
\ead{arobles@ugr.es}

\author[JCR]{J.C.~Rosales\fnref{fn1}}
\ead{jrosales@ugr.es}

\cortext[cor1]{Corresponding author}
\fntext[fn1]{Both authors are supported by the project MTM2014-55367-P, which is funded by Mi\-nis\-terio de Econom\'{\i}a y Competitividad and Fondo Europeo de Desarrollo Regional FEDER, and by the Junta de Andaluc\'{\i}a Grant Number FQM-343. The second author is also partially supported by the Junta de Andaluc\'{\i}a/Feder Grant Number FQM-5849.}

\address[AMRP]{Departamento de Matem\'atica Aplicada, Facultad de Ciencias, Universidad de Granada, 18071-Granada, Spain.}
\address[JCR]{Departamento de \'Algebra, Facultad de Ciencias, Universidad de Granada, 18071-Granada, Spain.}

\begin{abstract}
	We compute the Frobenius number for sequences of triangular and tetrahedral numbers. In addition, we study some properties of the numerical semigroups associated to those sequences.
\end{abstract}

\begin{keyword}
	Frobenius number \sep triangular numbers \sep tetrahedral numbers \sep telescopic sequences \sep free numerical semigroups.
	
	\MSC[2010] 11D07
\end{keyword}

\end{frontmatter}

\section{Introduction}

According to \cite{brauer}, Frobenius raised in his lectures the following question: given relatively prime positive integers $a_1,\ldots,a_n$, compute the largest natural number that is not representable as a non-negative integer linear combination of $a_1,\ldots,a_n$. Nowadays, it is known as the Frobenius (coin) problem. Moreover, the solution is called the Frobenius number of the set $\{a_1,\ldots,a_n\}$ and it is denoted by ${\mathrm F}(a_1, \ldots, a_n)$.

It is well known (see \cite{sylvester2, sylvester3}) that ${\mathrm F}(a_1,a_2)=a_1a_2-a_1-a_2$. However, at present, the Frobenius problem is open for $n \geq 3$. More precisely, Curtis showed in \cite{curtis} that it is impossible to find a polynomial formula (this is, a finite set of polynomials) that computes the Frobenius number if $n=3$. In addition, Ram\'{\i}rez Alfons\'{\i}n proved in \cite{alfonsinNP} that this problem is NP-hard for $n$ variables.

Many papers study particular cases (see \cite{alfonsin} for more details). Specially, when $\{a_1,\ldots,a_n\}$ is part of a ``classic'' integer sequences: arithmetic and almost arithmetic (\cite{brauer,roberts,lewin,selmer}), Fibonacci (\cite{marin-alfonsin-revuelta}), geometric (\cite{ong-ponomarenko}), Mersenne (\cite{mersenne}), repunit (\cite{repunit}), squares and cubes (\cite{squares-cubes,moscariello}), Thabit (\cite{thabit}), et cetera.

For example, in \cite{brauer} Brauer proves that
\begin{equation}\label{brauer1}
{\mathrm F}(n,n+1,\ldots,n+k-1)=\Big(\Big\lfloor \frac{n-2}{k-1} \Big\rfloor +1 \Big)n-1,
\end{equation}
where, if $x\in{\mathbb R}$, then $\lfloor x \rfloor \in {\mathbb Z}$ and $\lfloor x \rfloor\leq x < \lfloor x \rfloor+1$. On the other hand, denoting by $a(n)={\mathrm F}\Big(\frac{n(n+1)}{2},\frac{(n+1)(n+2)}{2},\frac{(n+2)(n+3)}{2}\Big)$, for $n\in{\mathbb N} \setminus \{0\}$, C. Baker conjectured that (see \url{https://oeis.org/A069755/internal})
\begin{equation}\label{baker1}
a(n) = \frac{-14 + 6(-1)^n + (3+9(-1)^n)n + 3(5+(-1)^n)n^2 + 6n^3}{8};
\end{equation}
\begin{equation}\label{baker2}
a(n) = \frac{6n^3 + 18n^2 + 12n - 8}{8}, \; \mbox{ for $n$ even};
\end{equation}
\begin{equation}\label{baker3}
a(n) = \frac{6n^3 + 12n^2 - 6n - 20}{8}, \; \mbox{ for $n$ odd}.
\end{equation}
Let us observe that both of these examples are particular cases of combinatorial numbers (or binomial coefficients) sequences, that is,
\begin{itemize}
	\item ${n \choose 1},{n+1 \choose 1},\ldots,{n+k-1 \choose 1}$ in the first case,
	\item ${n+1 \choose 2},{n+2 \choose 2},{n+3 \choose 2}$ in the second one.
\end{itemize}
Let us recall that ${n+1 \choose 2}$ is known as a \emph{triangular} (or \emph{triangle}) \emph{number} and that the \emph{tetrahedral numbers} correspond to ${n+2 \choose 3}$. These classes of numbers are precisely the aim of this paper.

In order to achieve our purpose, we use a well-known formula by Johnson (\cite{johnson}): if $a_1,a_2,a_3$ are relatively prime numbers and $\gcd\{a_1,a_2\}=d$, then
\begin{equation}\label{eq-johnson}
{\mathrm F}(a_1, a_2, a_3) = d{\mathrm F}\Big(\frac{a_1}{d},\frac{a_2}{d},a_3\Big) + (d-1)a_3.
\end{equation}
In fact, we use the well-known generalization by Brauer and Shockley (\cite{brauer-shockley}): if $a_1,\ldots,a_n$ are relatively prime numbers and $d=\gcd\{a_1,\ldots,a_{n-1}\}$, then
\begin{equation}\label{eq-brauer-shockley}
{\mathrm F}(a_1,\ldots,a_n) = d{\mathrm F}\Big(\frac{a_1}{d},\ldots,\frac{a_{n-1}}{d},a_n\Big) + (d-1)a_n.
\end{equation}
An interesting situation, to apply these formulae, corresponds with telescopic sequences (\cite{kirfel-pellikaan}) and leads to free numerical semigroups, which were introduced by Bertin and Carbonne (\cite{bertin-carbonne-1,bertin-carbonne-2}) and previously used by Watanabe (\cite{watanabe}). Let us note that this idea does not coincide with the categorical concept of a free object.

\begin{definition}\label{telescopic-sequence}
	Let $(a_1,\ldots,a_n)$ be a sequence of positive integers such that $\gcd\{a_1,\ldots,a_n\}=1$ (where $n\geq2$). Let $d_i=\gcd\{a_1,\ldots,a_i\}$ for $i=1,\ldots,n$. We say that $(a_1,\ldots,a_n)$ is a \emph{telescopic sequence} if $\frac{a_i}{d_i}$ is representable as a non-negative integer linear combination of $\frac{a_1}{d_{i-1}},\ldots,\frac{a_{i-1}}{d_{i-1}}$ for $i=2,\ldots,n$.
\end{definition}

Let us observe that, if $(a_1,\ldots,a_n)$ is a telescopic sequence, then the sequence $\big(\frac{a_1}{d_i},\ldots,\frac{a_i}{d_i}\big)$ is also telescopic for $i=2,\ldots,n-1$.

Let $({\mathbb N},+)$ be the additive monoid of non-negative integers. We say that $S$ is a \emph{numerical semigroup} if it is an additive subsemigroup of ${\mathbb N}$ which satisfies $0\in S$ and ${\mathbb N} \setminus S$ is a finite set.

Let $X=\{x_1,\ldots,x_n\}$ be a non-empty subset of ${\mathbb N} \setminus \{0\}$. We denote by $\langle X \rangle = \langle x_1,\ldots,x_n\rangle$ the monoid generated by $X$, that is,
$$\langle X \rangle=\{\lambda_1x_1+\cdots+\lambda_nx_n \mid \lambda_1,\ldots,\lambda_n\in{\mathbb N}\}.$$
It is well known (see \cite{springer}) that every submonoid $S$ of $({\mathbb N},+)$ has a unique \emph{minimal system of generators}, that is, there exists a unique $X$ such that $S=\langle X \rangle$ and $S\not=\langle Y \rangle$ for any $Y\subsetneq X$. In addition, $X$ is a system of generators of a numerical semigroup if and only if $\gcd(X)=1$.

Let $X=\{x_1,\ldots,x_n\}$ be the minimal system of generators of a numerical semigroup $S$. Then $n$ (that is, the cardinality of $X$) is called the \emph{embedding dimension} of $S$ and it is denoted by ${\mathrm e}(S)$.

\begin{definition}\label{free-sg}
	We say that $S$ is a \emph{free numerical semigroup} if there exists a telescopic sequence $(a_1,\ldots,a_n)$ such that $S=\langle a_1,\ldots,a_n\rangle$.
\end{definition}

Our purpose in this work is taking advantage of the notion of telescopic sequence in order to compute the Frobenius number associated to sequences of consecutive triangular (or tetrahedral) numbers. In order to achieve this purpose, let us observe two facts.
\begin{enumerate}
	\item Two consecutive triangular numbers are not relatively prime (Lemma~\ref{lem1}).
	\item It is easy to check that, if $n\geq 6$, then $\Big( {n+1 \choose 2},{n+2 \choose 2},{n+3 \choose 2},{n+4 \choose 2} \Big)$ is a sequence of four consecutive (relatively prime) triangular numbers but it does not admit any permutation which is telescopic.
\end{enumerate}
Therefore, we have to limit our study to sequences of three consecutive triangular numbers. In the same way, we have to take sequences of four consecutive tetrahedral numbers. In addition, it is possible to apply our techniques (in a much more tedious study) to sequences of five consecutive combinatorial numbers ${n \choose m}$ with $m=4$ (see Remark~\ref{rem-13}). However, if $m\geq5$, it is not possible to use such tools because, in general, we have not got telescopic sequences (see Remark~\ref{rem-14}).

Let us summarize the content of this work. In Section~\ref{tri-numbers} we compute the Frobenius number of three consecutive triangular numbers. In Section~\ref{tetra-numbers} we solve the analogue case for four consecutive tetrahedral numbers. In the last section, we show some results on numerical semigroups generated by three consecutive triangular numbers or four consecutive tetrahedral numbers, taking advantage of the fact that they are free numerical semigroups. Finally, point out that, to get a self-contained paper, we have included in Section~\ref{consequences} a preliminary subsection with backgrounds on minimal presentations, Ap\'ery sets, and Betti elements of a numerical semigroup.

\section{Triangular numbers}\label{tri-numbers}

Let us recall that a \emph{triangular number} (or \emph{triangle number}) is a positive integer which counts the number of dots composing an equilateral triangle. For example, in Figure~\ref{fig1} we show the first six triangular numbers.

\begin{figure}[ht]
	\centering
	\includegraphics[trim = 130pt 555pt 130pt 125pt, clip, width=1\textwidth]{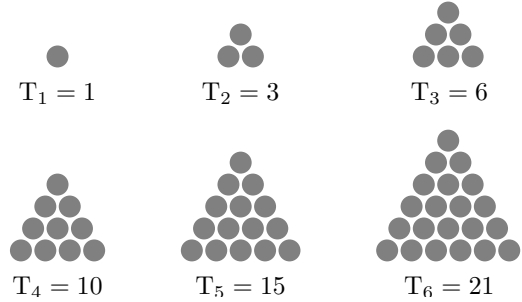}
	\caption{The first six triangular numbers}
	\label{fig1}
\end{figure}

It is well known that the $n$th triangular number is given by the combinatorial number ${\mathrm T}_n={n+1 \choose 2}$.

In order to compute the Frobenius number of a sequence of three triangular numbers, we need to determine if we have a sequence of relatively prime integers. First, we give a technical lemma.

\begin{lemma}\label{lem1}
	We have that
	$$\gcd \{{\mathrm T}_n,{\mathrm T}_{n+1}\} = \left\{ \begin{array}{l} \frac{n+1}{2}, \mbox{ if $n$ is odd;} \\[2pt] n+1, \mbox{ if $n$ is even.} \end{array} \right.$$
\end{lemma}

\begin{proof}
	If $n$ is odd, then we have that
	$$\gcd \{{\mathrm T}_n,{\mathrm T}_{n+1}\} = \gcd \left\{ \frac{n(n+1)}{2}, \frac{(n+1)(n+2)}{2}  \right\} = \frac{n+1}{2} \gcd\{n,2\} = \frac{n+1}{2}.$$
	On the other hand, if $n$ is even, then
	$$\gcd \{{\mathrm T}_n,{\mathrm T}_{n+1}\} = \gcd \left\{ \frac{n(n+1)}{2}, \frac{(n+1)(n+2)}{2}  \right\} =  (n+1) \gcd\left\{ \frac{n}{2}, 1 \right\} = n+1.$$
\end{proof}

In the following lemma we show that three consecutive triangular numbers are always relatively prime.

\begin{lemma}\label{lem2}
	$\gcd \{{\mathrm T}_n,{\mathrm T}_{n+1},{\mathrm T}_{n+2}\}=1.$
\end{lemma}

\begin{proof}
	By Lemma~\ref{lem1}, if $n$ is odd, then
	$$\gcd \{{\mathrm T}_n,{\mathrm T}_{n+1},{\mathrm T}_{n+2}\} = \gcd \big\{ \gcd \{ {\mathrm T}_n,{\mathrm T}_{n+1} \}, \gcd \{ {\mathrm T}_{n+1},{\mathrm T}_{n+2} \}  \big\} =$$ 
	$$\gcd \left\{  \frac{n+1}{2},n+2 \right\} = \gcd \left\{ \frac{n+1}{2}, \frac{n+1}{2} +1 \right\}=1.$$
	
	The proof is similar if $n$ is even. Therefore, we omit it. 
\end{proof}

In the next result, we show the key to obtain the answer to our question.

\begin{proposition}\label{prop3}
	The sequences $({\mathrm T}_n,{\mathrm T}_{n+1},{\mathrm T}_{n+2})$ and $({\mathrm T}_{n+2},{\mathrm T}_{n+1},{\mathrm T}_n)$ are telescopic.
\end{proposition}

\begin{proof}
	Let $n$ be an odd integer. From Lemmas~\ref{lem1} and \ref{lem2}, $\gcd\{{\mathrm T}_n,{\mathrm T}_{n+1}\}=\frac{n+1}{2}$ and $\gcd\{{\mathrm T}_n,{\mathrm T}_{n+1},{\mathrm T}_{n+2}\}=1$. Now, it is obvious that $$\frac{{\mathrm T}_{n+2}}{1} = \frac{n+3}{2}(n+2) \in \left\langle \frac{{\mathrm T}_n}{\,\frac{n+1}{2}\,}, \frac{{\mathrm T}_{n+1}}{\,\frac{n+1}{2}\,} \right\rangle = \langle n,n+2 \rangle.$$
	Therefore, $({\mathrm T}_n,{\mathrm T}_{n+1},{\mathrm T}_{n+2})$ is telescopic if $n$ is odd.

	Once again, from Lemmas~\ref{lem1} and \ref{lem2}, we have that $\gcd\{{\mathrm T}_{n+2},{\mathrm T}_{n+1}\}=n+2$ (observe that $n+1$ is even) and $\gcd\{{\mathrm T}_n,{\mathrm T}_{n+1},{\mathrm T}_{n+2}\}=1$. Then it is clear that $$\frac{{\mathrm T}_{n}}{1} = \frac{n}{2}(n+1) \in \left\langle \frac{{\mathrm T}_{n+2}}{\,n+2\,}, \frac{{\mathrm T}_{n+1}}{\,n+2\,} \right\rangle = \langle \frac{\,n+3\,}{2},{\,n+1\,}{2} \rangle.$$
	Thus, $({\mathrm T}_{n+2},{\mathrm T}_{n+1},{\mathrm T}_n)$ is telescopic if $n$ is odd. 
	 
	In a similar way, we can show that $({\mathrm T}_n,{\mathrm T}_{n+1},{\mathrm T}_{n+2})$ and $({\mathrm T}_{n+2},{\mathrm T}_{n+1},{\mathrm T}_n)$ are telescopic if $n$ is even.
\end{proof}

Now we are ready to give the main result of this section.

\begin{proposition}\label{prop5}
	Let $n \in {\mathbb N} \setminus \{0\}$. Then
	$${\mathrm F}({\mathrm T}_n,{\mathrm T}_{n+1},{\mathrm T}_{n+2})=\left\{ \begin{array}{l} \frac{3n^3+6n^2-3n-10}{4}, \mbox{ if $n$ is odd;} \\[3pt] \frac{3n^3+9n^2+6n-4}{4}, \mbox{ if $n$ is even.} \end{array} \right.$$
	Equivalently,
	\begin{equation}\label{eq-prop5}
	{\mathrm F}({\mathrm T}_n,{\mathrm T}_{n+1},{\mathrm T}_{n+2})= \left\lfloor \frac{n}{2} \right\rfloor ({\mathrm T}_n+{\mathrm T}_{n+1}+{\mathrm T}_{n+2}-1)-1.
	\end{equation}
\end{proposition}
	
\begin{proof}
	Let $n$ be an odd positive integer. From (\ref{eq-johnson}) (or (\ref{eq-brauer-shockley})) and the proof of Proposition~\ref{prop3}, we have that
	$${\mathrm F}({\mathrm T}_n,{\mathrm T}_{n+1},{\mathrm T}_{n+2})= \frac{n+1}{2} {\mathrm F} \left( \frac{{\mathrm T}_n}{\,\frac{n+1}{2}\,}, \frac{{\mathrm T}_{n+1}}{\,\frac{n+1}{2}\,}, {\mathrm T}_{n+2} \right) + \frac{n-1}{2}{\mathrm T}_{n+2}=$$
	$$\frac{n+1}{2} {\mathrm F} \left( n, n+2 \right) + \frac{n-1}{2}\frac{(n+2)(n+3)}{2}$$
	and, having in mind that ${\mathrm F} \left( n, n+2 \right)=n^2-2$, then the conclusion is obvious. On the other hand, the reasoning for even $n$ is similar. Finally, a straightforward computation leads to (\ref{eq-prop5}).
\end{proof}

\section{Tetrahedral numbers}\label{tetra-numbers}

Let us recall that a \emph{tetrahedral number} (or \emph{triangular pyramidal number}) is a positive integer which counts the number of balls composing a regular tetrahedron. The $n$th tetrahedral number is given by the combinatorial number $\mathrm{TH}_n={n+2 \choose 3}$. Thus, in Figure~\ref{fig2}, we see the pyramid (by layers) associated to the 5th tetrahedral number ($\mathrm{TH}_5=35$).

\begin{figure}[ht]
	\centering
	\includegraphics[trim = 135pt 557pt 135pt 125pt, clip, width=1\textwidth]{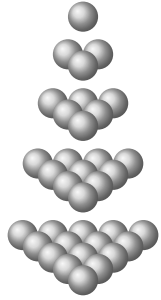}
	\caption{The tetrahedral number $\mathrm{TH}_5$ (representation by layers)}
	\label{fig2}
\end{figure}

In this section, our purpose is compute the Frobenius number for a sequence of four consecutive tetrahedral numbers.

We need a preliminary lemma with an immediate proof.

\begin{lemma}\label{lem-aure1}
	Let $(a_1,a_2,\ldots,a_n)$ be a sequence of positive integers such that $d_1=\gcd\{a_1,a_2,\ldots,a_n\}$. If $d_2=\gcd\{a_2-a_1,\ldots,a_n-a_{n-1}\}$, then $d_1 | d_2$. In particular, if $d_2=1$, then $d_1=1$.
\end{lemma}

Now, let us see that four consecutive tetrahedral numbers are always relatively prime.

\begin{lemma}\label{lem-aure2}
	$\gcd \{\mathrm{TH}_n,\mathrm{TH}_{n+1},\mathrm{TH}_{n+2},\mathrm{TH}_{n+3}\}=1.$	
\end{lemma}

\begin{proof}
	It is clear that
	$$(\mathrm{TH}_{n+1}-\mathrm{TH}_n, \mathrm{TH}_{n+2}-\mathrm{TH}_{n+1},\mathrm{TH}_{n+3}-\mathrm{TH}_{n+2})=({\mathrm T}_n, {\mathrm T}_{n+1}, {\mathrm T}_{n+2}).$$
	Therefore, by applying Lemmas~\ref{lem2} and \ref{lem-aure1}, we have the conclusion.
\end{proof}

The following lemma has an easy proof too. So, we omit it.

\begin{lemma}\label{lem-jc1}
	Let $n \in {\mathbb N}\setminus \{0\}$.
	\begin{enumerate}
		\item If $n=6k$, then $\gcd\left\{ \mathrm{TH}_n,\mathrm{TH}_{n+1} \right\} = (6k+1)(3k+1)$.
		\item If $n=6k+1$, then $\gcd\left\{ \mathrm{TH}_n,\mathrm{TH}_{n+1} \right\} = (3k+1)(2k+1)$.
		\item If $n=6k+2$, then $\gcd\left\{ \mathrm{TH}_n,\mathrm{TH}_{n+1} \right\} = (2k+1)(3k+2)$.
		\item If $n=6k+3$, then $\gcd\left\{ \mathrm{TH}_n,\mathrm{TH}_{n+1} \right\} = (3k+2)(6k+5)$.
		\item If $n=6k+4$, then $\gcd\left\{ \mathrm{TH}_n,\mathrm{TH}_{n+1} \right\} = (6k+5)(k+1)$.
		\item If $n=6k+5$, then $\gcd\left\{ \mathrm{TH}_n,\mathrm{TH}_{n+1} \right\} = (k+1)(6k+7)$.
	\end{enumerate}
\end{lemma}

In the next two results, we give the tool for getting the answer to our problem.

\begin{proposition}\label{prop-aure3} 
	The sequence $(\mathrm{TH}_n,\mathrm{TH}_{n+1},\mathrm{TH}_{n+2},\mathrm{TH}_{n+3})$ is telescopic if and only if $n\equiv r \bmod 6$ with $r\in\{0,1,2,3\}$.
\end{proposition}

\begin{proof}
	We are going to study the six possible cases $n=6k+r$ with $k \in {\mathbb N}$ and $r\in\{0,1,\ldots,5\}$.
	
	\begin{enumerate}
		\item Let $n=6k$. Since
		$\gcd\{\mathrm{TH}_n,\mathrm{TH}_{n+1},\mathrm{TH}_{n+2}\}$ is equal to
		$$\gcd\big\{ \gcd\{\mathrm{TH}_n,\mathrm{TH}_{n+1} \}, \gcd\{ \mathrm{TH}_{n+1},\mathrm{TH}_{n+2} \} \big\},$$
		from items 1 and 2 of Lemma~\ref{lem-jc1}, we have $\gcd\{\mathrm{TH}_n,\mathrm{TH}_{n+1},\mathrm{TH}_{n+2}\} = \gcd \{ (6k+1)(3k+1), (3k+1)(2k+1) \} = 3k+1.$
		Now, it is easy to check that $\mathrm{TH}_{n+3} = 0 \frac{\mathrm{TH}_n}{3k+1} + (3k+2) \frac{\mathrm{TH}_{n+1}}{3k+1} + 2 \frac{\mathrm{TH}_{n+2}}{3k+1}$.
		
		On the other hand, since
		$$\frac{(\mathrm{TH}_n,\mathrm{TH}_{n+1},\mathrm{TH}_{n+2})}{3k+1} = \big( 2k(6k+1),(6k+1)(2k+1),2(2k+1)(3k+2) \big),$$
		$\gcd\{2k(6k+1),(6k+1)(2k+1)\}=6k+1$, and $$2(2k+1)(3k+2) = 0 \frac{2k(6k+1)}{6k+1} + 2(3k+2) \frac{(6k+1)(2k+1)}{6k+1},$$ we conclude that $(\mathrm{TH}_n,\mathrm{TH}_{n+1},\mathrm{TH}_{n+2},\mathrm{TH}_{n+3})$ is telescopic.
		
		\item Having in mind items 2 and 3 of Lemma~\ref{lem-jc1}, if $n=6k+1$, then we get that $\gcd\{\mathrm{TH}_n,\mathrm{TH}_{n+1},\mathrm{TH}_{n+2}\}=2k+1$. In addition,
		$$\mathrm{TH}_{n+3} = 0 \frac{\mathrm{TH}_n}{2k+1} + 0 \frac{\mathrm{TH}_{n+1}}{2k+1} + 2(k+1) \frac{\mathrm{TH}_{n+2}}{2k+1}.$$
		Since $\gcd \left\{ \frac{\mathrm{TH}_n}{2k+1}, \frac{\mathrm{TH}_{n+1}}{2k+1} \right\} = 3k+1$ and
		$$\frac{\mathrm{TH}_{n+2}}{2k+1} = (3k+2)\frac{\mathrm{TH}_n}{(2k+1)(3k+1)} + 2 \frac{\mathrm{TH}_{n+1}}{(2k+1)(3k+1)},$$
		we have the result.
		
		\item For $n=6k+2$, we have that $\gcd\{\mathrm{TH}_n,\mathrm{TH}_{n+1},\mathrm{TH}_{n+2}\}=3k+2$,
		$$\mathrm{TH}_{n+3} = 0 \frac{\mathrm{TH}_n}{3k+2} + 3(k+1) \frac{\mathrm{TH}_{n+1}}{3k+2} + 2 \frac{\mathrm{TH}_{n+2}}{3k+2},$$
		$\gcd \left\{ \frac{\mathrm{TH}_n}{3k+2}, \frac{\mathrm{TH}_{n+1}}{3k+2} \right\} = 2k+1$, and 
		$$\frac{\mathrm{TH}_{n+2}}{3k+2} = 0\frac{\mathrm{TH}_n}{(3k+2)(2k+1)} + 2(k+1)\frac{\mathrm{TH}_{n+1}}{(3k+2)(2k+1)}.$$
		
		\item For $n=6k+3$, we have that $\gcd\{\mathrm{TH}_n,\mathrm{TH}_{n+1},\mathrm{TH}_{n+2}\}=6k+5$,
		$$\mathrm{TH}_{n+3} = 0 \frac{\mathrm{TH}_n}{6k+5} + 0 \frac{\mathrm{TH}_{n+1}}{6k+5} + 2(3k+4) \frac{\mathrm{TH}_{n+2}}{6k+5},$$
		$\gcd \left\{ \frac{\mathrm{TH}_n}{6k+5}, \frac{\mathrm{TH}_{n+1}}{6k+5} \right\} = 3k+2$, and 
		$$\frac{\mathrm{TH}_{n+2}}{6k+5} = 3(k+1)\frac{\mathrm{TH}_n}{(6k+5)(3k+2)} + 2\frac{\mathrm{TH}_{n+1}}{(6k+5)(3k+2)}.$$
		
		\item If $n=6k+4$, then $\gcd\{\mathrm{TH}_n,\mathrm{TH}_{n+1},\mathrm{TH}_{n+2}\}=k+1$. Let us suppose that there exist $\alpha,\beta,\gamma \in {\mathbb N}$ such that
		$\mathrm{TH}_{n+3} = \alpha \frac{\mathrm{TH}_n}{k+1} + \beta \frac{\mathrm{TH}_{n+1}}{k+1} + \gamma \frac{\mathrm{TH}_{n+2}}{k+1}$. Then
		$$(6k+7)\big((3k+4)(2k+3)-(6k+5)\beta - 2(3k+4)\gamma\big) = 2(3k+2)(6k+5)\alpha.$$
		Since $\gcd\{6k+7,2\}=\gcd\{6k+7,3k+2\}=\gcd\{6k+7,6k+5\}=1$, then there exists $\tilde{\alpha} \in {\mathbb N}$ such that $\alpha=(6k+7)\tilde{\alpha}$. Therefore,
		$$(3k+4)(2k+3)-(6k+5)\beta - 2(3k+4)\gamma = 2(3k+2)(6k+5)\tilde\alpha$$
		and, consequently,
		$$(3k+4) (2k+3-2\gamma) = (6k+5)\big(\beta+2(3k+2)\tilde\alpha\big).$$
		Thus, since $\gcd\{3k+4,6k+5\}=1$, we conclude that $(6k+5) \mid (2k+3-2\gamma)$ (that is, $6k+5$ divides $2k+3-2\gamma$) and, thereby, $2k+3-2\gamma=0$. Now, having in mind that $\gamma$ is a non-negative integer, the equality $2k+3=2\gamma$ is not possible. That is, we have a contradiction.
		
		\item If $n=6k+5$, then $\gcd\{\mathrm{TH}_n,\mathrm{TH}_{n+1},\mathrm{TH}_{n+2}\}=6k+7$,
		$$\mathrm{TH}_{n+3} = 0 \frac{\mathrm{TH}_n}{6k+7} + 0 \frac{\mathrm{TH}_{n+1}}{6k+7} + 2(3k+5) \frac{\mathrm{TH}_{n+2}}{6k+7},$$
		and $\gcd \left\{ \frac{\mathrm{TH}_n}{6k+7}, \frac{\mathrm{TH}_{n+1}}{6k+7} \right\} = k+1$. Let us suppose that there exist $\alpha,\beta \in {\mathbb N}$ such that $\frac{\mathrm{TH}_{n+2}}{6k+7} = \alpha\frac{\mathrm{TH}_n}{(6k+7)(k+1)} + \beta\frac{\mathrm{TH}_{n+1}}{(6k+7)(k+1)}$. In such a case,
		$$(3k+4)(2k+3) = (6k+5)\alpha + 2(3k+4)\beta.$$
		Now, since $\gcd\{3k+4,6k+5\}=1$, then $\alpha=(3k+4)\tilde{\alpha}$ for some $\tilde{\alpha} \in {\mathbb N}$ and, consequently, $2k+3-2\beta=(6k+5)\tilde{\alpha}$, that is, $(6k+5) \mid (2k+3-2\beta)$. Reasoning as in the previous case, since $\beta$ is a non-negative integer, we get a contradiction once again. \qedhere
	\end{enumerate}
\end{proof}

Using the same techniques as in the previous proof, we have the next result.

\begin{proposition}\label{prop-aure4}
	The sequence $(\mathrm{TH}_{n+3},\mathrm{TH}_{n+2},\mathrm{TH}_{n+1},\mathrm{TH}_n)$ is telescopic if and only if $n\equiv r \bmod 6$ with $r\in\{4,5\}$.
\end{proposition}

By combining Propositions~\ref{prop-aure3} and \ref{prop-aure4} with (\ref{eq-brauer-shockley}), it is clear that we can obtain the Frobenius number for every sequence $(\mathrm{TH}_n,\mathrm{TH}_{n+1},\mathrm{TH}_{n+2},\mathrm{TH}_{n+3})$. Thus we get the following result.

\begin{proposition}\label{prop5b}
	Let $n \in {\mathbb N}\setminus \{0\}$. Then ${\mathrm F}\left( \mathrm{TH}_n,\mathrm{TH}_{n+1},\mathrm{TH}_{n+2},\mathrm{TH}_{n+3} \right) =$
	\begin{enumerate}
		\item $\frac{n-3}{3}\mathrm{TH}_{n+1} + n\mathrm{TH}_{n+2} + \frac{n}{2}\mathrm{TH}_{n+3} - \mathrm{TH}_n$, if $n=6k$;
		\item $(n-1)\mathrm{TH}_{n+1} + \frac{n-1}{2}\mathrm{TH}_{n+2} + \frac{n-1}{3}\mathrm{TH}_{n+3} - \mathrm{TH}_n$, if $n=6k+1$;
		\item $(n-1)\mathrm{TH}_{n+1} + \frac{n-2}{3}\mathrm{TH}_{n+2} + \frac{n}{2}\mathrm{TH}_{n+3} - \mathrm{TH}_n$, if $n=6k+2$;
		\item $\frac{n-3}{3}\mathrm{TH}_{n+1} + \frac{n-1}{2}\mathrm{TH}_{n+2} + (n+1)\mathrm{TH}_{n+3} - \mathrm{TH}_n$, if $n=6k+3$;
		\item $\frac{n+2}{3}\mathrm{TH}_{n+2} + \frac{n+2}{2}\mathrm{TH}_{n+1} + (n+2)\mathrm{TH}_n - \mathrm{TH}_{n+3}$, if $n=6k+4$;
		\item $(n+4)\mathrm{TH}_{n+2} + \frac{n+1}{3}\mathrm{TH}_{n+1} + \frac{n+1}{2}\mathrm{TH}_n - \mathrm{TH}_{n+3}$, if $n=6k+5$.
	\end{enumerate}
\end{proposition}

\begin{remark}\label{rem-13}
	From the contents of this section and the previous one, it looks like that the problem becomes longer and longer when we consider the sequences $\Big( {n+m-1 \choose m}, \ldots, {n+2m-1 \choose m} \Big)$ with $m$ increasing (and $n$ fixed). Anyway, it is easy to see that,
	\begin{itemize}
		\item if $n\geq1$, then $\Big( {n+3 \choose 4}, {n+4 \choose 4}, {n+5 \choose 4}, {n+6 \choose 4}, {n+7 \choose 4} \Big)$ is a telescopic sequence if and only if $n\equiv x \bmod 6$ for $x \in \{0,1,2\}$;
		\item if $n\geq9$, then $\Big( {n+7 \choose 4}, {n+6 \choose 4}, {n+5 \choose 4}, {n+4 \choose 4}, {n+3 \choose 4} \Big)$ is telescopic if and only if $n\equiv x \bmod 6$ for $x \in \{3,4,5\}$;
		\item if $n\in\{3,4,5\}$, then both of above sequences are telescopic.
	\end{itemize}
	Therefore, we can give a general formula for the Frobenius problem associated to five consecutive combinatorial numbers given of the type ${n \choose 4}$. (In order to study this case, it is better to consider $n\equiv x \bmod 12$ for $x \in \{0,1,\ldots,11 \}$.)
\end{remark}

\begin{remark}\label{rem-14}	
	Let $n,m$ be positive integers. At this moment, we could conjecture that the sequence $\Big( {n+m-1 \choose m}, \ldots, {n+2m-1 \choose m} \Big)$ is telescopic if and only if the sequence $\Big( {n+2m-1 \choose m}, \ldots, {n+m-1 \choose m} \Big)$ is not telescopic and, consequently, we would have an easy algorithmic process to compute ${\mathrm F}\Big( {n+m-1 \choose m}, \ldots, {n+2m-1 \choose m} \Big)$. Unfortunately, neither $\Big( {12 \choose 5}, \ldots, {17 \choose 5} \Big) = (792,1287,2002,3003,4368,6188)$ nor $(6188,4368,3003,2002,1287,792)$ are telescopic. In fact, all possible permutations of $(792,1287,2002,3003,4368,6188)$ are not telescopic.
\end{remark}

\section{Consequences on numerical semigroups}\label{consequences}

As we comment in the introduction, if $(a_1,\ldots,a_n)$ is a sequence of relatively prime positive integers, then the monoid $\langle a_1,\ldots,a_n \rangle$ is a numerical semigroup. In this section we are interested in those numerical semigroups which are generated by three consecutive triangular numbers or by four consecutive tetrahedral numbers and have embedding dimension equal to three or four, respectively. In fact, having in mind that
\begin{itemize}
	\item $({\mathrm T}_n,{\mathrm T}_{n+1},{\mathrm T}_{n+2})$ is always telescopic,
	\item either $(\mathrm{TH}_n,\mathrm{TH}_{n+1},\mathrm{TH}_{n+2},\mathrm{TH}_{n+3})$ or $(\mathrm{TH}_{n+3},\mathrm{TH}_{n+2},\mathrm{TH}_{n+1},\mathrm{TH}_n)$ is telescopic,
\end{itemize}
we can use the ideas of \cite[Chapter~8]{springer} to obtain several results for the numerical semigroups ${\mathcal T}_n=\langle {\mathrm T}_n,{\mathrm T}_{n+1},{\mathrm T}_{n+2} \rangle$ and $\mathcal{TH}_n=\langle \mathrm{TH}_n,\mathrm{TH}_{n+1},\mathrm{TH}_{n+2},\mathrm{TH}_{n+3} \rangle$.

Firstly, we compute the embedding dimension of ${\mathcal T}_n$ and $\mathcal{TH}_n$.

\begin{lemma}\label{lem-aure41}
	We have that ${\mathrm e}({\mathcal T}_1)=1$, ${\mathrm e}({\mathcal T}_2)=2$, ${\mathrm e}({\mathcal T}_3)=3$ for all $n\geq 3$, ${\mathrm e}(\mathcal{TH}_1)=1$, ${\mathrm e}(\mathcal{TH}_2) = {\mathrm e}(\mathcal{TH}_3)=3$, ${\mathrm e}(\mathcal{TH}_n)=4$ for all $n\geq 4$.
\end{lemma}

\begin{proof}
	It is obvious that
	\begin{itemize}
		\item ${\mathcal T}_1 = \langle 1,3,6 \rangle = \langle 1 \rangle = {\mathbb N}$;
		\item ${\mathcal T}_2 = \langle 3,6,10 \rangle = \langle 3,10 \rangle$;
		\item $\mathcal{TH}_1 = \langle 1,4,10,20 \rangle = \langle 1 \rangle = {\mathbb N}$;
		\item $\mathcal{TH}_2 = \langle 4,10,20,35 \rangle = \langle 4,10,35 \rangle$;
		\item $\mathcal{TH}_3 = \langle 10,20,35,56 \rangle = \langle 10,35,56 \rangle$.
	\end{itemize}
	Thereby, ${\mathrm e}({\mathcal T}_1)=1$, ${\mathrm e}({\mathcal T}_2)=2$, ${\mathrm e}(\mathcal{TH}_1)=1$, and ${\mathrm e}(\mathcal{TH}_2) = {\mathrm e}(\mathcal{TH}_3)=3$.
	
	Now let ${\mathcal T}_n = \langle \frac{n(n+1)}{2}, \frac{(n+1)(n+2)}{2}, \frac{(n+2)(n+3)}{2} \rangle$ with $n\geq 3$. Then, $\frac{n(n+1)}{2} < \frac{(n+1)(n+2)}{2} < \frac{(n+2)(n+3)}{2}$, $\frac{n(n+1)}{2} \nmid \frac{(n+1)(n+2)}{2}$ (that is, $\frac{n(n+1)}{2}$ does not divide $\frac{(n+1)(n+2)}{2}$) and $\frac{n(n+1)}{2} \nmid \frac{(n+2)(n+3)}{2}$.
	
	On the other hand, if we suppose that $\frac{(n+2)(n+3)}{2} \in \langle \frac{n(n+1)}{2}, \frac{(n+1)(n+2)}{2} \rangle$, then there exist $\alpha,\beta \in {\mathbb N}$ such that $\frac{(n+2)(n+3)}{2} = \alpha\frac{n(n+1)}{2} + \beta \frac{(n+1)(n+2)}{2}$. Consequently, $(n+2)(n+3) = \alpha n(n+1) + \beta (n+1)(n+2)$, that is, $n+1$ should be a divisor of $(n+2)(n+3)$, which is a contradiction.
	
	Thus, we conclude that $\left\{ \frac{n(n+1)}{2}, \frac{(n+1)(n+2)}{2}, \frac{(n+2)(n+3)}{2} \right\}$ is a minimal system of generators of ${\mathcal T}_n$ for all $n\geq 3$ and, in consequence, ${\mathrm e}({\mathcal T}_3)=3$ for all $n\geq 3$. 
	
	By using similar arguments, we prove that ${\mathrm e}(\mathcal{TH}_n)=4$ for all $n\geq 4$.
\end{proof}

Since we want to consider general cases, along this section we are going to take ${\mathcal T}_n$ with $n\geq 3$ and $\mathcal{TH}_n$ with $n\geq4$. Thus, we always have that ${\mathrm e}({\mathcal T}_n)=3$ and ${\mathrm e}(\mathcal{TH}_n)=4$. The remaining five cases are left as exercises to the reader.

\subsection{Preliminaries}

The concepts developed below can be generalized to more general frameworks. We refer to \cite{presentations} (and the references therein) for more details.

Let $S$ be a numerical semigroup and let $\{n_1,\ldots,n_e\}$ be the minimal system of generators of $S$. Then we define the \emph{factorization homomorphism associated to $S$}
$$\varphi_S : {\mathbb N}^e \to S, \quad u=(u_1,\ldots,u_e) \mapsto \textstyle{\sum^e_{i=1}}u_in_i.$$
Let us observe that, if $s\in S$, then the cardinality of $\varphi_S^{-1}(s)$ is just the number of factorizations of $s$ in $S$.

We define the \emph{kernel congruence of $\varphi_S$} on ${\mathbb N}^e$ as follows,
$$u \sim_S v \mbox{ if } \varphi_S(u)=\varphi_S(v).$$
It is well known that $S$ is isomorphic to the monoid ${\mathbb N}^e/\!\sim_S$.

Let $\rho \in {\mathbb N}^e \times {\mathbb N}^e$. Then the intersection of all congruences containing $\rho$ is the so-called \emph{congruence generated by $\rho$}. On the other hand, if $\sim$ is the congruence generated by $\rho$, then we say that $\rho$ is a \emph{system of generators} of $\sim$. In this way, a \emph{presentation} of $S$ is a system of generators of $\sim_S$, and a \emph{minimal presentation} of $S$ is a minimal system of generators of $\sim_S$. Let us observe that, for numerical semigroups, the concepts of minimal presentation with respect cardinality and with respect set inclusion coincide (see \cite[Corollary~1.3]{algoritmo-relaciones} or \cite[Corollary~8.13]{springer}).

Since every numerical semigroup $S$ is finitely generated, it follows that $S$ has a minimal presentation with finitely many elements, that is, $S$ is \emph{finitely presented} (see \cite{redei}). In addition, all minimal presentations of a numerical semigroup $S$ have the same cardinality (see \cite[Corollary~1.3]{algoritmo-relaciones} or \cite[Corollary~8.13]{springer}). 

Let us describe an algorithmic process (see \cite{algoritmo-relaciones} or \cite{springer}) to compute all the minimal presentations of $S$.

If $s\in S$, then we define over $\varphi_S^{-1}(s)$ the binary relation ${\mathcal R}_s$ as follows: for $u=(u_1,\ldots,u_e),v=(v_1,\ldots,v_e) \in \varphi_S^{-1}(s)$, we say that $u {\mathcal R}_s v$ if there exists a chain $u_0,u_1,\ldots,u_r \in \varphi_S^{-1}(s)$ such that $u_0=u$, $u_r=v$, and $u_i\cdot u_{i+1}\not= 0$ for all $i\in \{0,\ldots,r-1\}$ (where $\cdot$ is the usual element-wise product of vectors).

If $\varphi_S^{-1}(s)$ has a unique ${\mathcal R}_s$-class, then we take the set $\rho_s=\emptyset$. Otherwise, if ${\mathcal R}_{s,1},\ldots,{\mathcal R}_{s,l}$ are the different classes of $\varphi_S^{-1}(s)$, then we choose $v_i \in {\mathcal R}_{s,i}$ for all $i\in\{1,\ldots,l\}$ and consider $\rho_s$ to be any set of $k-1$ pairs of elements in $V=\{v_1,\ldots,v_l\}$ such that any two elements in $V$ are connected by a sequence of pairs in $\rho_s$ (or their symmetrics).

Finally, the set $\rho = \cup_{s\in S} \rho_s$ is a minimal presentation of $S$.

From the previous comments, it is clear that for each numerical semigroup $S$ there are finitely many elements $s\in S$ such that $\varphi_S^{-1}(s)$ has more than one ${\mathcal R}_s$-class. Precisely, such elements are known as the \emph{Betti elements} of $S$ (see \cite{gs-o}).

\begin{remark}\label{rem-aure39}
	Let us observe that Betti elements are known from a minimal presentation. In fact, the Betti elements of a numerical semigroup $S$ can be computed as the evaluation of the elements of a minimal presentation of $S$.
\end{remark}

Let us define some constants which allow for characterizing the free numerical semigroups and, moreover, getting easily minimal presentations of such numerical semigroups.

\begin{definition}
	Let $S$ be a numerical semigroup with minimal system of generators $\{n_1,\ldots,n_e\}$. For every $i\in \{2,\ldots,e \}$, we define
	$$c^*_i=\min\big\{k\in{\mathbb N} \mid kn_i \in \langle n_1,\ldots,n_{i-1} \rangle \big\}.$$
\end{definition}

For our purposes, we consider the following characterization for free numerical semigroups (see \cite[Proposition~9.15]{springer} and the comment after it).

\begin{lemma}\label{lem-aure43}
	Let $S$ be a numerical semigroup with minimal system of generators $\{n_1,\ldots,n_e\}$. Then $S$ is free (for the arrangement $\{n_1,\ldots,n_e \}$) if and only if $n_1=c^*_2\cdots c^*_e$.
\end{lemma}

In order to have a minimal presentation of a free numerical semigroup $S$, we use the following result from \cite{springer}.

\begin{lemma}[Corollary 9.18]\label{lem-aure45}
	Let $S$ be a free numerical semigroup for the arrangement of its minimal set of generators $\{n_1,\ldots,n_e \}$. Assume that $c^*_in_i=a_{i_1}n_1+\cdots+a_{i_{i-1}}n_{i-1}$ for some $a_{i_1},\ldots,a_{i_{i-1}} \in {\mathbb N}$. Then $$\left\{(c^*_ix_i,a_{i_1}x_1+\cdots+a_{i_{i-1}}x_{i-1}) \mid i \in \{2,\ldots,e\} \right\}$$ is a minimal presentation of $S$.
\end{lemma}

From Remark~\ref{rem-aure39} and Lemma~\ref{lem-aure45}, we can easily compute the Betti elements of a free numerical semigroup $S$.

\begin{lemma}\label{lem-aure40}
	Let $S$ be a free numerical semigroup for the arrangement of its minimal set of generators $\{n_1,\ldots,n_e \}$. Assume that $c^*_in_i=a_{i_1}n_1+\cdots+a_{i_{i-1}}n_{i-1}$ for some $a_{i_1},\ldots,a_{i_{i-1}} \in {\mathbb N}$. Then the Betti elements of $S$ are the numbers $c^*_in_i$ with $i \in \{2,\ldots,e\}$.
\end{lemma}

Let us recall that, if $S$ is a numerical semigroup and $n\in S\setminus \{0\}$, then the \emph{Ap\'ery set of $n$ in $S$} (see \cite{apery}) is
$$\mathrm{Ap}(S,n) = \left\{ s\in S \mid s-n \notin S \right\} = \left\{0,\omega(1),\ldots,\omega(n-1) \right\},$$
where $\omega(i)$, $1\leq i \leq n-1$, is the least element of $S$ congruent with $i$ modulo $n$. Let us observe that the cardinality of $\mathrm{Ap}(S,n)$ is just equal to $n$. In addition, it is clear that ${\mathrm F}(S)=\max\big(\mathrm{Ap}(S,n)\big)-n$.

When $S$ is a free numerical semigroup for the arrangement $\{n_1,\ldots,n_e \}$, then we can compute explicitly (and easily) the set $\mathrm{Ap}(S,n_1)$. (The next lemma is part of \cite[Lemma~9.15]{springer}, where there is not a detailed proof. By following the ideas exposed in \cite[Lemma~9.14]{springer}, we show a full proof.)

\begin{lemma}\label{lem-aure46}
	Let $S= \langle n_1,\ldots,n_e \rangle$ a free numerical semigroup. Then
	$$\mathrm{Ap}(S,n_1)=\big\{\lambda_2n_2+\cdots+\lambda_en_e \mid \lambda_j\in \{0,\ldots,c^*_j-1\} \mbox{ for all } j \in \{2,\ldots,e\}  \big\}.$$
\end{lemma}

\begin{proof}
	Let $\omega(i)$ be a non-zero element of $\mathrm{Ap}(S,n_1)$. Since $\omega(i) \in S$, then there exist $\alpha_1,\ldots,\alpha_e \in {\mathbb N}$ such that $\omega(i)=\alpha_1n_1+\cdots+\alpha_en_e$. From the definition of Ap\'ery set, it is clear that $\alpha_1=0$ and, consequently, $\omega(i)=\alpha_2n_2+\cdots+\alpha_en_e$.
	
	Now, let us take $c^*_e$. Then, $\alpha_e = \gamma_e c^*_e + \delta_e$ with $\gamma_e \in {\mathbb N}$ and $0\leq \delta_e < c^*_e$. And, since $c^*_e=a_{e_1}n_1+\cdots+a_{e_{e-1}}n_{e-1}$ for some $a_{e_1},\ldots,a_{e_{e-1}} \in {\mathbb N}$, then $\omega(i)=\beta_1n_1+\cdots+\beta_{e-1}n_{e-1}+\delta_en_e$ with $\beta_2,\ldots,\beta_{e-1} \in {\mathbb N}$ and $\beta_1=0$ (remember that $\omega(i)\in \mathrm{Ap}(S,n_1)$). Repeating this process with the coefficient of $n_{e-1}$ up to that of $n_2$, we have that $\omega(i)=\delta_2n_2+\cdots+\delta_en_e$ with $0\leq \delta_j < c^*_j$ for all $j\in \{2,\ldots,e\}$.
	
	At this moment, we have that
	\begin{equation}\label{eq-1}
	\mathrm{Ap}(S,n_1) \subseteq \big\{\lambda_2n_2+\cdots+\lambda_en_e \mid \lambda_j\in \{0,\ldots,c^*_j\} \mbox{ for all } j \in \{2,\ldots,e\}  \big\}.
	\end{equation}
	In order to finish the proof, it is enough to observe that the cardinality of $\mathrm{Ap}(S,n_1)$ is $n_1$ and that, from Proposition~\ref{lem-aure43}, the cardinality of the second set in~\ref{eq-1} is least than or equal to $c^*_2\ldots c^*_e=n_1$.
\end{proof}

\subsection{Triangular case}\label{subsect-tri}

Let us take ${\mathcal T}_n=\langle {\mathrm T}_n, {\mathrm T}_{n+1}, {\mathrm T}_{n+2} \rangle$ with $n\geq 3$. We begin computing the values of $c^*_2$ and $c^*_3$ of ${\mathcal T}_n$ (for the arrangement $\{{\mathrm T}_n, {\mathrm T}_{n+1}, {\mathrm T}_{n+2}\}$).

\begin{lemma}\label{lem-aure42}
	Let ${\mathcal T}_n=\langle {\mathrm T}_n, {\mathrm T}_{n+1}, {\mathrm T}_{n+2} \rangle$. Then
	\begin{enumerate}
		\item $c^*_2=n$ and $c^*_3=\frac{n+1}{2}$ if $n$ is odd;
		\item $c^*_2=\frac{n}{2}$ and $c^*_3=n+1$ if $n$ is even.
	\end{enumerate}
\end{lemma}

\begin{proof}
	Let $n$ be odd. By definition, there exist $\alpha, \beta \in {\mathbb N}$ such that $c^*_3 {\mathrm T}_{n+2} = \alpha {\mathrm T}_n + \beta {\mathrm T}_{n+1}$, that is,
	$$c^*_3 \frac{(n+2)(n+3)}{2} = \alpha \frac{n(n+1)}{2} + \beta \frac{(n+1)(n+2)}{2}.$$
	From Lemma~\ref{lem1}, we have that $\gcd\left\{\frac{n(n+1)}{2}, \frac{(n+1)(n+2)}{2} \right\} = \frac{n+1}{2}$. Moreover, $\gcd\left\{n+2,\frac{n+1}{2} \right\} = \gcd\left\{\frac{n+3}{2},\frac{n+1}{2} \right\}=1$. Therefore, $c^*_3$ has to be a multiply of $\frac{n+1}{2}$. Since
	$$\frac{n+1}{2} \frac{(n+2)(n+3)}{2} = \frac{n+3}{2} \frac{n(n+1)}{2} + 0 \frac{(n+1)(n+2)}{2},$$
	we conclude that $c^*_3=\frac{n+1}{2}$.
	
	Now, we have that $c^*_2 {\mathrm T}_{n+1} = \alpha {\mathrm T}_n$ for some $\alpha \in {\mathbb N}$. That is, $c^*_2 \frac{(n+1)(n+2)}{2} = \alpha \frac{n(n+1)}{2}$ or, equivalently, $c^*_2(n+2)=\alpha n$. Then, it is obvious that $c^*_2=n$.
	
	Repeating similar arguments, we have the result in the case of even $n$.
\end{proof}

\begin{remark}
	An analogous result is obtained if we consider the arrangement $\{{\mathrm T}_{n+2}, {\mathrm T}_{n+1}, {\mathrm T}_n\}$ for the minimal system of generators of ${\mathcal T}_n$ (see Proposition~\ref{prop3}). In fact,
	\begin{enumerate}
		\item $c^*_2=\frac{n+3}{2}$ and $c^*_3=n+2$ if $n$ is odd;
		\item $c^*_2=n+3$ and $c^*_3=\frac{n+2}{2}$ if $n$ is even.
	\end{enumerate}
\end{remark}

From Lemmas~\ref{lem-aure41}, \ref{lem-aure43}, and \ref{lem-aure42}, we have the next result.

\begin{proposition}\label{prop-aure44}
	${\mathcal T}_n$ is a free numerical semigroup with embedding dimension equal to three.
\end{proposition}

By combining Proposition~\ref{prop-aure44}, Lemma~\ref{lem-aure45} and the arguments in the proof of Lemma~\ref{lem-aure42}, we get a minimal presentation of ${\mathcal T}_n$. 

\begin{proposition}\label{prop6}
	A minimal presentation of $\,{\mathcal T}_n = \langle {\mathrm T}_n, {\mathrm T}_{n+1}, {\mathrm T}_{n+2} \rangle$ is
	\begin{enumerate}
		\item if $n$ is odd: $\left\{ \big(\frac{n+1}{2}x_3,\frac{n+3}{2}x_2\big), \big(nx_2,(n+2)x_1\big) \right\}$;
		\item if $n$ is even: $\left\{ \big((n+1)x_3,(n+3)x_2\big), \big(\frac{n}{2}x_2,\frac{n+2}{2}x_1\big) \right\}$.
	\end{enumerate}
\end{proposition}

By applying Lemma~\ref{lem-aure40}, we compute the Betti elements of ${\mathcal T}_n$.

\begin{corollary}\label{cor8}
	The Betti elements of ${\mathcal T}_n$ are
	\begin{enumerate}
		\item $\frac{3}{2}{n+3 \choose 3}$ and $3{n+2 \choose 3}$, if $n$ is odd;
		\item $3{n+3 \choose 3}$ and $\frac{3}{2}{n+2 \choose 3}$, if $n$ is even.
	\end{enumerate}
\end{corollary}

\begin{remark}
	Let us observe that the Betti elements of ${\mathcal T}_n$ are given in terms of tetrahedral numbers. An analogue property can be observed in the case of $\mathcal{TH}_n$ (see Corollary~\ref{cor-aure7}): the Betti elements of $\mathcal{TH}_n$ can be expressed in terms of the combinatorial numbers ${p \choose 4}$.
\end{remark}

Finally, taking in Lemma\ref{lem-aure46} the values of $c^*_2$ and $c^*_3$ given in Lemma~\ref{lem-aure42}, we obtain the explicit description of the set $\mathrm{Ap}({\mathcal T}_n,{\mathrm T}_n)$.

\begin{corollary}\label{cor7}
	We have that
	\begin{enumerate}
		\item if $n$ is odd, then
		$$\mathrm{Ap}({\mathcal T}_n,{\mathrm T}_n)=\Bigg\{ a{\mathrm T}_{n+1}+b{\mathrm T}_{n+2} \;\Big|\; a\in\{0,\ldots,n-1\}, \; b\in\left\{0,\ldots,\frac{n-1}{2}\right\} \Bigg\};$$
		\item if $n$ is even, then
		$$\mathrm{Ap}({\mathcal T}_n,{\mathrm T}_n)=\Bigg\{ a{\mathrm T}_{n+1}+b{\mathrm T}_{n+2} \;\Big|\; a\in\left\{0,\ldots,\frac{n-2}{2}\right\}, \; b\in\{0,\ldots,n\} \Bigg\}.$$
	\end{enumerate}
\end{corollary}

\begin{remark}\label{rem-aure5c}
	Having in mind that, if $S$ is a numerical semigroup and $n\in S\setminus \{0\}$, then ${\mathrm F}(S)=\max\big(\mathrm{Ap}(S,n)\big)-n$, we recover Proposition~\ref{prop5} (for $n\geq3$).
\end{remark}

\subsection{Tetrahedral case}\label{subsect-tetra}

We finish with the serie of results relative to the numerical semigroups $\mathcal{TH}_n$ with $n\geq4$. Since the reasonings and tools are similar to the used ones in the triangular case (and some ones already seen in Section~\ref{tetra-numbers}), we omit the proofs.

\begin{lemma}\label{lem-aure51a}
	Let $\mathcal{TH}_n=\langle \mathrm{TH}_n, \mathrm{TH}_{n+1}, \mathrm{TH}_{n+2}, \mathrm{TH}_{n+3} \rangle$, where $n\equiv r \bmod 6$ with $r\in\{0,1,2,3\}$. Then
	\begin{enumerate}
		\item $c^*_2=\frac{n}{3}$, $c^*_3=n+1$, and $c^*_4=\frac{n+2}{2}$ if $n=6k$;
		\item $c^*_2=n$, $c^*_3=\frac{n+1}{2}$, and $c^*_4=\frac{n+2}{3}$ if $n=6k+1$;
		\item $c^*_2=n$, $c^*_3=\frac{n+1}{3}$, and $c^*_4=\frac{n+2}{2}$ if $n=6k+2$;
		\item $c^*_2=\frac{n}{3}$, $c^*_3=\frac{n+1}{2}$, and $c^*_4=n+2$ if $n=6k+3$.
	\end{enumerate}
\end{lemma}

\begin{lemma}\label{lem-aure51b}
	Let $\mathcal{TH}_n=\langle \mathrm{TH}_{n+3}, \mathrm{TH}_{n+2}, \mathrm{TH}_{n+1}, \mathrm{TH}_n \rangle$, where $n\equiv r \bmod 6$ with $r\in\{4,5\}$. Then
	\begin{enumerate}
		\item $c^*_2=\frac{n+5}{3}$, $c^*_3=\frac{n+4}{2}$, and $c^*_4=n+3$ if $n=6k+4$;
		\item $c^*_2=n+5$, $c^*_3=\frac{n+4}{3}$, and $c^*_4=\frac{n+3}{2}$ if $n=6k+5$.
	\end{enumerate}
\end{lemma}

\begin{proposition}\label{prop-aure52}
	$\mathcal{TH}_n$ is a free numerical semigroup with embedding dimension equal to four.
\end{proposition}

\begin{proposition}\label{prop-aure5}
	A minimal presentation of $\mathcal{TH}_n$ is
	\begin{enumerate}
		\item if $n=6k$: $${\textstyle \left\{ \Big(\frac{n+2}{2}x_4,2x_3+\frac{n+4}{2}x_2\Big), \big((n+1)x_3,(n+4)x_2\big), \Big(\frac{n}{3}x_2,\frac{n+3}{3}x_1\Big) \right\};}$$
		
		\item if $n=6k+1$: $${\textstyle \left\{ \Big(\frac{n+2}{3}x_4,\frac{n+5}{3}x_3\Big), \Big(\frac{n+1}{2}x_3,2x_2+\frac{n+3}{2}x_1\Big), \big(nx_2,(n+3)x_1\big) \right\};}$$
		
		\item if $n=6k+2$: $${\textstyle \left\{ \Big(\frac{n+2}{2}x_4,2x_3+\frac{n+4}{2}x_2\Big), \Big(\frac{n+1}{3}x_3,\frac{n+4}{3}x_2\Big), \big(nx_2,(n+3)x_1\big) \right\};}$$
		
		\item if $n=6k+3$: $${\textstyle 4\left\{ \big((n+2)x_4,(n+5)x_3\big), \Big(\frac{n+1}{2}x_3,2x_2+\frac{n+3}{2}x_1\Big), \Big(\frac{n}{3}x_2,\frac{n+3}{3}x_1\Big) \right\};}$$
		
		\item if $n=6k+4$: $${\textstyle \left\{ \big((n+3)x_1,nx_2\big), \Big(\frac{n+4}{2}x_2,\frac{n-1}{3}x_3+\frac{n+2}{6}x_4\Big), \Big(\frac{n+5}{3}x_3,\frac{n+2}{3}x_4\Big) \right\}};$$
		
		\item if $n=6k+5$: $${\textstyle \left\{ \Big(\frac{n+3}{2}x_1,\frac{n-2}{3}x_2+\frac{n+1}{6}x_3\Big), \Big(\frac{n+4}{3}x_2,\frac{n+1}{3}x_3\Big), \big((n+5)x_3,(n+2)x_4\big) \right\}.}$$
	\end{enumerate}
\end{proposition}

\begin{corollary}\label{cor-aure7}
	The Betti elements of $\mathcal{TH}_n$ are
	\begin{enumerate}
		\item $2{n+5 \choose 4}$, $4{n+4 \choose 4}$ and $\,\frac{4}{3}{n+3 \choose 4}$, if $n=6k$;
		\item $\frac{4}{3}{n+5 \choose 4}$, $2{n+4 \choose 4}$ and $\,4{n+3 \choose 4}$, if $n=6k+1$;
		\item $2{n+5 \choose 4}$, $\frac{4}{3}{n+4 \choose 4}$ and $\,4{n+3 \choose 4}$, if $n=6k+2$;
		\item $4{n+5 \choose 4}$, $2{n+4 \choose 4}$ and $\,\frac{4}{3}{n+3 \choose 4}$, if $n=6k+3$;
		\item $\frac{4}{3}{n+5 \choose 4}$, $2{n+4 \choose 4}$ and $\,4{n+3 \choose 4}$, if $n=6k+4$;
		\item $4{n+5 \choose 4}$, $\frac{4}{3}{n+4 \choose 4}$ and $\,2{n+3 \choose 4}$, if $n=6k+5$.
	\end{enumerate}
\end{corollary}

\begin{corollary}\label{cor-aure6}
	We have that
	\begin{enumerate}
		\item if $n=6k$, then
		$$\mathrm{Ap}(\mathcal{TH}_n,\mathrm{TH}_n)=\Bigg\{ a\mathrm{TH}_{n+1}+b\mathrm{TH}_{n+2}+c\mathrm{TH}_{n+3} \;\Big|\; a\in\left\{0,\ldots,\frac{n-3}{3}\right\},$$ 
		$$b\in\{0,\ldots,n\}, \; c\in\left\{0,\ldots,\frac{n}{2}\right\} \Bigg\};$$
		\item if $n=6k+1$, then
		$$\mathrm{Ap}(\mathcal{TH}_n,\mathrm{TH}_n)=\Bigg\{ a\mathrm{TH}_{n+1}+b\mathrm{TH}_{n+2}+c\mathrm{TH}_{n+3} \;\Big|\; a\in\{0,\ldots,n-1\},$$ 
		$$b\in\left\{0,\ldots,\frac{n-1}{2}\right\}, \; c\in\left\{0,\ldots,\frac{n-1}{3}\right\} \Bigg\};$$
		\item if $n=6k+2$, then
		$$\mathrm{Ap}(\mathcal{TH}_n,\mathrm{TH}_n)=\Bigg\{ a\mathrm{TH}_{n+1}+b\mathrm{TH}_{n+2}+c\mathrm{TH}_{n+3} \;\Big|\; a\in\{0,\ldots,n-1\},$$ 
		$$b\in\left\{0,\ldots,\frac{n-2}{3}\right\}, \; c\in\left\{0,\ldots,\frac{n}{2}\right\} \Bigg\};$$
		\item if $n=6k+3$, then
		$$\mathrm{Ap}(\mathcal{TH}_n,\mathrm{TH}_n)=\Bigg\{ a\mathrm{TH}_{n+1}+b\mathrm{TH}_{n+2}+c\mathrm{TH}_{n+3} \;\Big|\; a\in\left\{0,\ldots,\frac{n-3}{3}\right\},$$ 
		$$b\in\left\{0,\ldots,\frac{n-1}{2}\right\}, \; c\in\{0,\ldots,n+1\} \Bigg\};$$
		\item if $n=6k+4$, then
		$$\mathrm{Ap}(\mathcal{TH}_n,\mathrm{TH}_{n+3})=\Bigg\{ a\mathrm{TH}_n+b\mathrm{TH}_{n+1}+c\mathrm{TH}_{n+2} \;\Big|\; a\in\{0,\ldots,n+2\},$$ 
		$$b\in\left\{0,\ldots,\frac{n+2}{2}\right\}, \; c\in\left\{0,\ldots,\frac{n+2}{3}\right\} \Bigg\};$$
		\item if $n=6k+5$, then
		$$\mathrm{Ap}(\mathcal{TH}_n,\mathrm{TH}_{n+3})=\Bigg\{ a\mathrm{TH}_n+b\mathrm{TH}_{n+1}+c\mathrm{TH}_{n+2} \;\Big|\; a\in\left\{0,\ldots,\frac{n+1}{2}\right\},$$ 
		$$b\in\left\{0,\ldots,\frac{n+1}{3}\right\}, \; c\in\{0,\ldots,n+4\} \Bigg\}.$$
	\end{enumerate}
\end{corollary}

\begin{remark}
	We can get Proposition~\ref{prop5b} (for $n\geq4$) as an immediate consequence of Corollary~\ref{cor-aure6} (see Remark~\ref{rem-aure5c}).
\end{remark}

\begin{remark}
	Having in mind  Remark~\ref{rem-13}, it is possible to developed a serie of results, which are analogue to those obtain in Subsections~\ref{subsect-tri} and \ref{subsect-tetra}, for sequences of the type $\Big( {n+3 \choose 4}, {n+4 \choose 4}, {n+5 \choose 4}, {n+6 \choose 4}, {n+7 \choose 4} \Big)$ with $n\geq 6$.
\end{remark}

%

%

\end{document}